\documentclass[reqno]{amsart}

\usepackage{amsmath,amssymb,amsthm}

\usepackage{tikz}
\usetikzlibrary{calc,intersections,through,backgrounds}
\usepackage[latin1]{inputenc}

\usepackage{caption, hyperref}
\usepackage{graphicx}
\usepackage{graphics}

\newtheorem*{thm}{Theorem}

\newtheorem*{corollary}{Corollary}

\begin{document}

\title[]{Regularizing random points\\ by deleting a few}

\author[]{Dmitriy Bilyk}
\address{School of Mathematics, University of Minnesota, Minneapolis, MN 55455, USA}
\email{dbilyk@umn.edu}

\author[]{Stefan Steinerberger}
\address{Department of Mathematics, University of Washington, Seattle, WA 98195, USA} \email{steinerb@uw.edu}


\begin{abstract} It is well understood that if one is given a set $X \subset [0,1]$ of $n$ independent uniformly distributed random variables, then
$$ \sup_{0 \leq x \leq 1} \left| \frac{\# X \cap [0,x]}{\# X} - x \right| \lesssim \frac{\sqrt{\log{n}}}{ \sqrt{n}} \qquad \mbox{with very high probability.} $$
We show that one can improve the error term by removing a few of the points. For any $m \leq 0.001n$ there exists a subset $Y \subset X$ obtained by deleting at most $m$ points, so that the error term drops from $\sim \sqrt{\log{n}}/\sqrt{n}$ to $ \log{(n)}/m$ with high probability. When $m=cn$ for a small $0 \leq c \leq 0.001$, this achieves the essentially optimal asymptotic order of discrepancy $\log(n)/n$. The proof is constructive and works in an online setting (where one is given the points sequentially, one at a time, and has to decide whether to keep or discard it). A change of variables shows the same result for any random variables on the real line with absolutely continuous density.
\end{abstract}
\maketitle

\section{Introduction and Result}
\subsection{Problem.} Suppose $x_1, \dots, x_n$ are $n$ independent uniformly distributed random variables on the unit interval $[0,1]$. This is one of the most fundamental settings in probability theory which we understand well. In particular, for $0 \leq x \leq 1$
$$ \# \left\{ 1 \leq i \leq n: x_i \leq x \right\} = n \cdot x \pm \mathcal{O}(\sqrt{n}) \qquad \mbox{with high likelihood.}$$
More precisely, the quantity on the left is a binomial random variable $\mathcal{B}(n, x)$ and is well approximated by a Gaussian distribution for large $n$. Moreover, the quantity $ \# \left\{ 1 \leq i \leq n: x_i \leq x \right\} - n x$ behaves, for large $n$, like a (scaled) Brownian bridge. Is it possible to increase the rate at which $n$ randomly sampled points approximate the uniform distribution by deleting a few, say up to $m \leq \varepsilon n$, of the points? 

\begin{center}
\begin{figure}[h!]
        \begin{tikzpicture}
        \node at (0,0) {\includegraphics[width=0.4\textwidth]{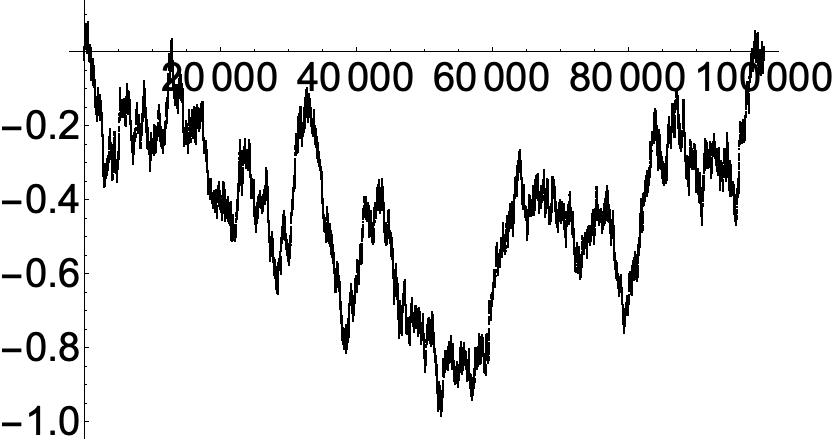}};
        \node at (6,0) {\includegraphics[width=0.4\textwidth]{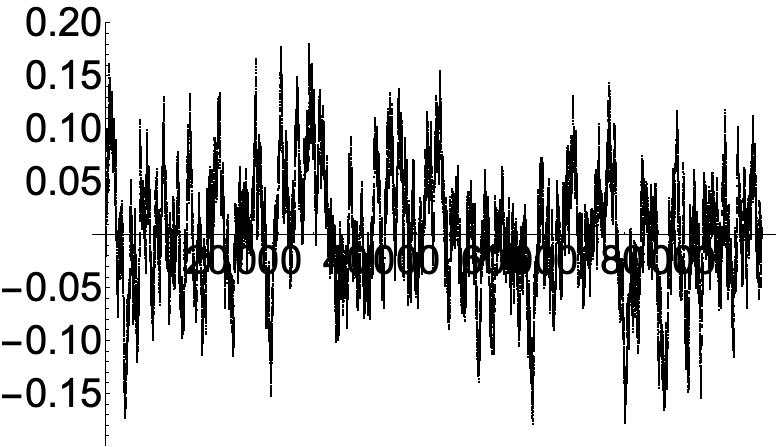}};
    \end{tikzpicture}
    \caption{Left: $n=10^5$ points and $\sqrt{\# X} \cdot (x_i - i/\# X)$. Right: $\sqrt{\# Y} \cdot (y_i - i/\# Y)$ where $Y \subset X$ and $\# Y \geq 0.95 \cdot \# X$. By dropping $5\%$ of the points, one can dramatically increase regularity.}
\end{figure}
\end{center}

\subsection{Result.} 
All results are stated in terms of the Kolmogorov-Smirnov statistic 
\begin{equation}\label{e.KS}
F_n(\left\{ x_1, \dots, x_n \right\}) = \max_{0 \leq x \leq 1} \left| \frac{\# (\left\{x_1, \dots, x_n\right\} \cap [0,x])}{n} - x  \right|.
\end{equation}
It is well-known that if the $x_i$ are $n$ independent uniformly distributed random variables, then
$$ \sqrt{n} \cdot  F_n(\left\{ x_1, \dots, x_n \right\}) \quad  \underbrace{\longrightarrow}_{n \rightarrow \infty} \quad \sup_{0 \leq t \leq 1} |B(t)|,$$
where $B$ is the Brownian bridge on $[0,1]$. This supremum follows the Kolmogorov distribution which is well understood. In particular, we have 
$$ \mathbb{P}\left(\sup_{0 \leq t \leq 1} |B(t)| \geq z \right) = 2 \sum_{k=1}^{\infty} (-1)^{k-1} e^{-2 k^2 z^2}$$
and the likelihood of it attaining large values is very small.
Our result says that by deleting a small sample of size at most $m$, we can decrease the Kolmogorov-Smirnov statistic to $(\log{n})/m$. The result is only interesting when $m \gtrsim \sqrt{n \log{n}}$: if $m$ is smaller,  the original set already has the desired property with high likelihood.
In discrepancy theory, the Kolmogorov-Smirnov statistic \eqref{e.KS} for (deterministic) sequences $( x_k )_{k=1}^\infty$ in $[0,1]$ is known as the \textit{star-discrepancy}. Going back to the seminal work of Weyl \cite{weyl}, it is one of the main objects of study in the field and quantifies the equidistribution of a sequence.

\begin{thm}
    Let $X = \left\{x_1, \dots, x_n \right\}$ be independent uniformly distributed random variables on $[0,1]$. For any $0 \leq m \leq 0.001n$, there exists 
$Y \subset X$ with
  $$ \# Y \geq \# X - m$$ 
    such that, with high likelihood,
    \begin{equation}\label{e.main}
    \max_{0 \leq x \leq 1} \left| \frac{\# Y\cap [0,x]}{\# Y} - x  \right| \leq 100 \frac{\log{n}}{m}.
    \end{equation}
 If $m=cn$ for some $c\in (0,0.001]$, this bound is 
        \begin{equation}\label{e.main1}
 \max_{0 \leq x \leq 1} \left| \frac{\# Y\cap [0,x]}{\# Y} - x  \right| \leq  \frac{100}{c} \frac{\log{n}}{n} \sim \frac{\log \# Y}{\# Y}.\end{equation}
\end{thm}

Here and in the rest of the paper, by ``high probability'', we mean that the probability of the complement is of the order $n^{-C}$. The power $C$ can be made arbitrarily large by adjusting the constants in the arguments.  
The result is optimal up to the value of the constant when $m=0.001n$. Given $n$ i.i.d. uniform random variables, the longest interval $J$ not containing any points has expected size $\sim \log{n}/n$ (see \cite{holst}) and thus no better result is possible.  
Another indication of optimality comes from discrepancy theory: a classical result of Schmidt \cite{schmidt} states that for \textit{any} sequence of points  $( x_k )_{k=1}^\infty$ in $[0,1]$ the inequality $$F_n(\left\{ x_1, \dots, x_n \right\}) \ge C \frac{\log n}{n}  $$ (with an absolute constant $C>0$) holds for infinitely many values of $n$ and the bound is sharp. In general, the bound,  of course, cannot  hold for \textit{all} $n \in N $: indeed, if $\{x_1,\dots,x_n\}$ are equispaced points for some $n$, then $F_n \sim 1/n$, but such behavior cannot be expected of random points. Hence, \eqref{e.main1} states that removing a tiny portion of random points achieves the optimal order of  discrepancy for sequences.
The proof is constructive and can be turned into a procedure which can be carried out in the `online' setting where one is presented with the $n$ random variables $\left\{x_1, \dots, x_n \right\}$ one at a time and has to decide immediately (before seeing any of the remaining points) whether to accept or reject the point. The argument generalizes to real-valued random variables $X$ whose distribution function is absolutely continuous.

\begin{corollary}
    Let $X = \left\{x_1, \dots, x_n \right\}$ be independent real-valued random variables from a common absolutely continuous probability measure with (continuous) cdf $F(x)$. For any $0 \leq m \leq 0.001n$, there exists 
$Y \subset X$ with
    $$ \# Y \geq \# X - m$$ 
    such that, with high likelihood,
    $$\sup_{x \in \mathbb{R}} \left| \frac{\# Y\cap [-\infty,x]}{\# Y} - F(x)  \right| \leq 100 \frac{\log{n}}{m}.$$
\end{corollary}
\begin{proof}
If $x_1, \dots, x_n$ are iid samples from a random variable with (continuous) cdf $F(x)$, then $F(x_1), F(x_2), \dots, F(x_n)$ are iid random variables following a uniform distribution on $[0,1]$. We note that
$$\frac{\# (Y\cap [-\infty,x])}{\# Y} = \frac{\#( \left\{F(x_1), \dots, F(x_n) \right\} \cap [0,F(x)])}{\# Y}$$
and from this the result follows.
\end{proof}

\subsection{Related results.} Despite the simplicity of the statement, we were unable to find it in the literature. Our main motivation was a result of Dwivedi, Feldheim, Gurel-Gurevich and Ramdas \cite{dwi} who proved that, given an infinite sequence of i.i.d. uniform random variables $(x_k)_{k=1}^{\infty}$ in the unit cube $[0,1]^d$, there is a way of discarding no more than $50\%$ of the points to reduce the discrepancy (with respect to all axis-parallel rectangles) from $\Theta(\sqrt{ \log \log (n)/n})$ to $\mathcal{O}(\log^{2d+1}{(n)}/n)$. When $d=1$, we are improving the bound from $\mathcal{O}(\log^3{(n)}/n)$ to $\mathcal{O}(\log{(n)}/n)$ which is sharp. \\
Another motivation was a recent result of Cl\'ement, Doerr and Paquete \cite{cle} who showed empirically  (among other things) that starting with a random set of points one can greatly improve the discrepancy by removing relatively few points. This mirrors another numerical observation \cite{bilyk} showing that, again empirically, one can often decrease the discrepancy dramatically by adding a relatively small number of points. Our argument is relatively simple but relies heavily on the one-dimensional setup; it is not at all clear how one would generalize the argument to higher dimensions. One would naturally expect analogous statements of this type to be true. We believe this to be an interesting area for future investigations.

\section{Proof}
\begin{proof}
The proof comes in three steps.

\begin{enumerate}
    \item \textit{Binning.} We split the unit interval into $k$ intervals $J_1, J_2, \dots, J_k$ of length $1/k$ each. We expect that each interval contains approximately $n/k$ points. Large deviation estimates show that, with high probability, each of the interval receives at least $n/k - \lambda$ points where $\lambda \sim \sqrt{ n \log{(n)} /k}$.
    \item \textit{Thinning.} Using $X$ to denote the $n$ random points, we consider each subinterval $J_{\ell}$ and delete $\# (X \cap J_{\ell}) - \left( n/k - \lambda\right)$ points uniformly at random. This leaves us with $k$ intervals each containing exactly $n/k - \lambda$ points. 
    \item \textit{Bounding.} It remains to bound the Kolmogorov-Smirnov statistic of the remaining set. Whenever $x = \ell/k$ for some integer $0 \leq \ell \leq k$, then the discrepancy is 0. In between the intervals, the
    randomness of the construction ensures independence and we can apply the Dvoretzky-Kiefer-Wolfowitz inequality $k$ times in combination with the union bound. 
\end{enumerate}

\begin{center}
    \begin{figure}[h!]
        \centering
        \begin{tikzpicture}
            \node at (0,0) {   \includegraphics[width=0.5\linewidth]{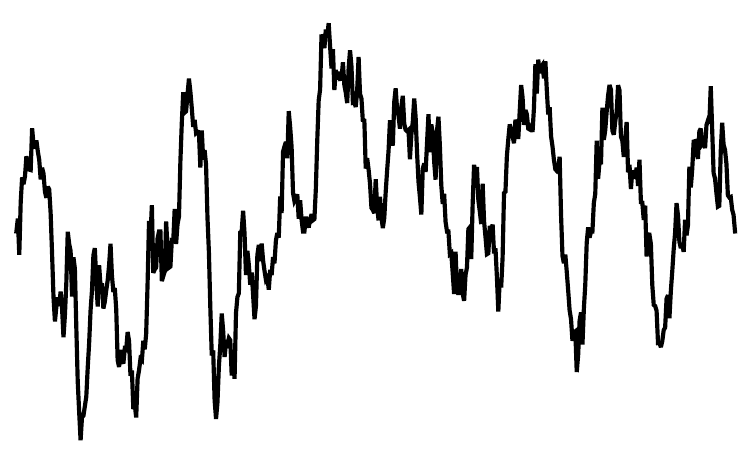}};
            \draw [ultra thick] (-3,-0) -- (3,-0);
            \filldraw (-3, -0) circle (0.15cm);
           \filldraw (-3+1.2, -0) circle (0.15cm);
            \filldraw (-3+2*1.2, -0) circle (0.15cm);
             \filldraw (-3+3*1.2, -0) circle (0.15cm);
              \filldraw (-3+4*1.2, -0) circle (0.15cm);
               \filldraw (-3+5*1.2, -0) circle (0.15cm);
        \end{tikzpicture}    
        \caption{The idea behind the proof visualized: by binning and thinning, the distribution function can be rewritten as the concatenation of independent distribution function that achieve discrepancy 0 at equispaced intervals. The continuum limit are concatenated Brownian bridges.}
        \label{fig:enter-label}
    \end{figure}
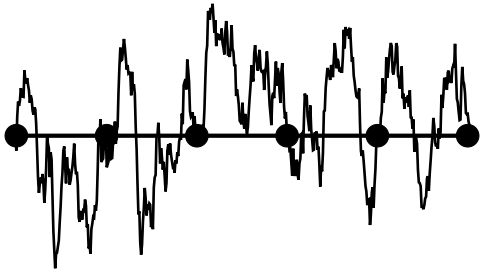
\end{center}

\subsection{ Binning.}
We partition $[0,1]$ into $k$ intervals
$$ [0,1] = \bigcup_{\ell=1}^{k} J_{\ell} \qquad \qquad \mbox{where} \quad J_{\ell} = \left[ \frac{\ell -1}{k}, \frac{\ell}{k} \right)$$
and $J_{k} = [(k-1)/k, 1]$. Here, $1 \leq k \leq 0.001n/\log{n}$ is a parameter to be optimized over later. Given $n$ identically and uniformly distributed random variables, the number of points
expected to land in each interval is $n/k$. We pick a parameter $\lambda > 0$ and will start by showing that for a suitable $\lambda$ (depending on $n$ and $k$), every single one of the $k$ interval receives at least $n/k - \lambda$ points with high probability. The number of points that land in a fixed interval is a Bernoulli random variable $\mathcal{B}(n,p)$ where $ p = 1/k$. We can write such a Bernoulli random variable as
$$X = X_1 + \dots + X_n$$ 
where the $n$ random variables $X_i \in \left\{0,1\right\}$ are independent and satisfy
$$ \mathbb{P}(X_i = 1) = \frac{1}{k} \qquad \mbox{and} \qquad  \mathbb{P}(X_i = 1) = 1-\frac{1}{k}.$$
We use the following form of the (multiplicative) Chernoff bound \cite[Theorem 4.5]{mitz}: if $X_1, \dots, X_n$ are independent random variables taking values in $\left\{0,1\right\}$ and $X$ denotes their sum and $\mu = \mathbb{E}X$, then, for any $0 < \delta < 1$,
$$ \mathbb{P}\left( X \leq (1-\delta)\mu \right) \leq \exp \left(-\frac{\delta^2 \mu}{2}\right).$$
Applying this bound with $\delta = \lambda/(n/k)$, we obtain
\begin{align*}
    \mathbb{P}\left(X \leq \frac{n}{k} - \lambda\right) &\leq \exp\left( -\frac{\lambda^2}{2n(n/k)} \right) =\exp\left( -\frac{\lambda^2 k}{2n} \right).
\end{align*}
In particular, for 
$$ \lambda = \sqrt{\frac{10 n \log{n}}{k}},$$
we have
$$ \mathbb{P}\left(X \leq \frac{n}{k} - \lambda\right) \leq \frac{1}{n^{5}}$$
from which we deduce with the union bound and $k \leq n $
$$ \mathbb{P}\left(\exists~k:  \# X \cap J_k \leq \frac{n}{k} - \lambda\right) \leq \frac{1}{n^4}.$$
We want to make sure that we never discard too many points, we always want to ensure that
$$ \lambda \leq \frac{1}{10} \frac{n}{k} \qquad \mbox{which is implied by} \qquad k \leq \frac{1}{1000} \frac{n}{\log{n}}.$$

\subsection{Thinning.}
Let us assume now that $X = \left\{x_1, \dots, x_n\right\}$ is a set of $n$ independently and uniformly distributed points with the property that
$$ \forall~ 1 \leq \ell \leq k \qquad \qquad \# (X \cap J_{\ell}) \geq \frac{n}{k} - \lambda,$$
where $\lambda = \sqrt{10 n \log{(n)} / k}$ as above (as we just demonstrated this property holds with high likelihood). We select a subset $Y \subset X$ as follows: for each interval $J_{\ell}$ we remove a fixed number of points until the number of points in the interval is exactly $\left\lfloor n/k - \lambda \right\rfloor$.  
More precisely, in each interval $J_{\ell}$ we select a random subset $X \cap J_{\ell}$ of size $\# (X \cap J_{\ell}) - \left\lfloor (n/k - \lambda)\right\rfloor$ and 
remove them. This random procedure is independent across different intervals.
This leaves us with a total of
\begin{equation*}\label{e.M} M =  \# Y = k \left \lfloor n/k - \lambda \right\rfloor \geq n - k \lambda -k  = n - \sqrt{10 n k \log{n} } -k 
\end{equation*}
points. As already mentioned before, the assumption $k \leq 0.001 n/\log{n}$ ensured that the number of remaining points is large, say $M \geq n/2$. We also observe that, by construction, each of the $k$ intervals now has the exact same number of points and, for each $1 \leq \ell \leq k$,
$$ \# \left(Y \cap \left[0, \frac{\ell}{k} \right]\right) = \frac{\ell}{k} \cdot \# Y,$$
therefore, for $x= \ell/k$, we have $$\frac{\# Y\cap [0,x]}{\# Y} - x = 0.$$

\subsection{Bounding.}
The purpose of this section is to now derive an upper bound for 
$$ \mbox{the quantity} \quad \max_{0 \leq x \leq 1} \left| \frac{\# Y\cap [0,x]}{\# Y} - x  \right|.$$
By construction, the quantity is 0 whenever $x$ is the right endpoint of any of the $k$ intervals. This suggests, for $0 < x < 1$ arbitrary, the decomposition
$$ x= \frac{\ell}{k} + b \qquad \mbox{where} \quad \ell \in \mathbb{N} \quad \mbox{and}\quad  0 \leq b < \frac{1}{k}.$$
Then
\begin{align*}
 \# (Y \cap \left[0, x \right]) - x \cdot \#Y &=  \# \left(Y \cap \left[\frac{\ell}{k}, b \right]\right) - \left( x - \frac{\ell}{k}\right) \cdot \# Y \\
 &= \# \left(Y \cap \left[\frac{\ell}{k}, b \right]\right) - b \cdot \# Y.
 \end{align*}
We observe that this can be interpreted as a random variable that only depends on the distribution of points in the subinterval $J_{\ell}$ and is independent of all the other subintervals. The question simplifies to having $k$ (independent) copies of the following problem: given a set $Z \subset [0,1/k]$ of $\# Z = M/k$ independently and uniformly distributed random points in $[0,1/k]$, what can we say about
\begin{equation}\label{e.residue} \mbox{the size of the quantity} \quad \sup_{0 \leq x \leq \frac{1}{k}} \left| \frac{1}{M} \left( \# Z \cap [0,x] \right) -  x \right| \quad ?\end{equation}
A change of variables shows that, using the notation $kZ = \left\{kz: z \in Z \right\}$ to refer to the rescaled set. Since $Z$ is a set of $M/k$ uniformly distributed random points in $[0,1/k]$, the set $kZ$ is a set of $M/k$ uniformly distributed random points in $[0,1]$. We compute, for $0 \leq x \leq 1/k$,
\begin{align*}
    \frac{1}{M} \left( \# Z \cap [0,x] \right) -  x &= \frac{1}{k} \left(    \frac{k}{M} \#\left(  Z \cap [0,x] \right) - k  x \right) =  \frac{1}{k} \left(    \frac{k}{M} \# \left( kZ \cap [0,kx] \right) -  k  x \right)\\
        &=  \frac{1}{k} \left(    \frac{ \# \left( kZ \cap [0,kx] \right)}{M/k} -  k  x \right) = \frac{1}{k} \left(   \frac{  \# \left(kZ \cap [0,kx] \right)}{\# kZ} -  k  x \right)\\
        &= \frac{1}{k} \left(   \frac{ \#  \left( kZ \cap [0,y] \right)}{\# kZ} -  y \right).
\end{align*}
This can be seen as a self-improved argument: we are back at the original problem, but with a smaller number of points, namely $\# kZ = M/k$, and an additional factor of $1/k$ in front.
 The Dvoretzky-Kiefer-Wolfowitz inequality states that, for a set of $Z$ of $n$ independent and uniformly distributed random variables on $[0,1]$
$$ \mathbb{P}\left( \sup_{0 \leq x \leq 1} \left| \frac{1}{n} \left( \# Z \cap [0,x] \right) -  x \right| > \varepsilon \right) \leq 2 e^{-2 n \varepsilon^2}.$$
This inequality was first discovered \cite{dvor} with an implicit constant $C>0$ on the right-hand side. It was later shown in \cite{massart} that the sharp constant is $C=2$ verifying a conjecture of Birnbaum-McCarty \cite{BM}. Therefore, in our setting, with $Z \subset [0,1/k]$ of $\# Z = M/k$ points and with the above rescaing, we have that
$$ \mathbb{P} = \mathbb{P} \left( \sup_{0 \leq x \leq \frac{1}{k}} \left| \frac{1}{M} \left( \# Z \cap [0,x] \right) -  x \right| > \varepsilon\right)$$
can be written as
\begin{align}
\mathbb{P} = \mathbb{P} \left( \sup_{0 \leq x \leq 1} \left|  \frac{\# kZ \cap [0,x]}{\# kZ} -  x \right| > k\varepsilon\right) \leq 2 e^{-2 M k \varepsilon^2}.
\end{align}
Setting
$$ \varepsilon = 10 \frac{\sqrt{\log{(Mk)}}}{\sqrt{Mk}},$$
we see that
$$ \mathbb{P}\left( \max_{0 \leq x \leq \frac{1}{k}} \left| \frac{1}{M} \left( \# Z \cap [0,x] \right) -  x \right| \geq   10 \frac{\sqrt{\log{(Mk)}}}{\sqrt{Mk}}\right) \leq 2 (Mk)^{-200}. $$
Using that $k \geq 1$ and $M \geq n/2$ together with the union bound, we arrive at
$$ \mathbb{P}\left( \max_{0 \leq x \leq 1} \left| \frac{\# Y\cap [0,x]}{\# Y} - x  \right| \geq   10 \frac{\sqrt{\log{(Mk)}}}{\sqrt{Mk}} \right) \leq \frac{c}{n^{100}}.$$

\subsection{Conclusion.}
We conclude by summarizing the preceding steps. We have, by deleting points in each of the $k$ intervals, constructed a set $Y$
with
$$ M  = \#Y \geq n  - \sqrt{10 nk \log{n}} -k  \geq \frac{n}{2}   \qquad \mbox{points}.$$
The previous subsection showed that this set $Y$ is fairly regular and satisfies
$$ \mathbb{P}\left( \max_{0 \leq x \leq 1} \left| \frac{\# Y\cap [0,x]}{\# Y} - x  \right| \geq   10 \frac{\sqrt{\log{(Mk)}}}{\sqrt{Mk}} \right) \leq \frac{c}{n^{100}}.$$
Using $Mk \leq n^2$, we can simplify this algebraically a little bit: since $\sqrt{\log{(Mk)}} \leq 2 \sqrt{\log{n}}$, we can conclude that with very high likelihood
$$  \max_{0 \leq x \leq 1} \left| \frac{\# (Y \cap [0,x]}{\# Y} - x \right| \leq  20 \frac{\sqrt{\log{(n )}}}{\sqrt{nk}}.$$
It remains to algebraically rephrase this is a little in terms of the number of discarded points. Recall that $m =  k\lambda = \sqrt{10 nk \log{(n)} }$ is an upper bound on the number of discarded points. Since $k \geq 1$, we have
 $m \geq \sqrt{10 n \log{n}}$.
Since we also require $k \leq 0.001n/\log{n}$ (to ensure that many points are left over, i.e. $M \geq n/2$), we have
 $m \leq 0.1n$. With these bounds, we can further rewrite the upper bound on the Kolmogorov-Smirnov statistic  
$$  20 \frac{\sqrt{\log{(n )}}}{\sqrt{nk}} = 20 \frac{\sqrt{\log{(n )}}}{m} \sqrt{10 \log{n}} \leq 80 \frac{\log{n}}{m}.$$
This proves  the desired statement \eqref{e.main} in the regime $m \geq \sqrt{10 n \log{n}}$. Observe that inequality \eqref{e.main1} corresponds to taking $k\sim n/\log(n)$.
To finish the proof of the main Theorem, it remains to deal with the case when $m \leq \sqrt{10 n \log{n}}$. In this  case we aim to remove at most $m$ points to achieve discrepancy
$$  100 \frac{\log{n}}{m} \geq 100 \frac{\log{n}}{ \sqrt{10 n \log{n}} } \geq 10 \frac{\sqrt{\log{n}}}{\sqrt{n}}.$$ 
This can be done by removing no points: the Dvoretzky-Kiefer-Wolfowitz inequality implies that the original set already achieves this with high probability.
\end{proof}

\textbf{Remark.} One can think of the proof as constructive, indeed, the procedure can be carried out in an online setting: knowing $n$ and $m \in [\sqrt{10 n \log{n}}, 0.1n] $  in advance (in the other regime there are no points to discard) allows us to compute $k$ and $\lambda$. We then simply stop accepting a new random point $x_i \in J_{\ell}$ as soon as the total number of variables already in $J_{\ell}$ reaches $\lceil n/k - \lambda \rceil$. \\

\textbf{Remark.} The binning procedure can  also  be viewed as being equivalent to a variant of the so-called \textit{jittered sampling}, see e.g. \cite{PauSt}: indeed, it amounts to dividing the domain $[0,1]$ into $k$ equal intervals and placing $n- k\lambda$ independent random points into each of them. However, this point of view would not simplify the proof. In most situations (which are less delicate than ours) when this construction is applied, the residual (boundary) terms are estimated using large deviation, e.g. Chernoff bounds. In the present case, this essentially amounts to answering the question posed in \eqref{e.residue}, which seems to require the full power of the Dvoretzky-Kiefer-Wolfowitz inequality.\\

\textbf{Acknowledgments.} The authors were supported by the NSF DMS-2054606 (DB) and DMS-2123224 (SS). They are grateful to the  Soci\'et\'e nationale des chemins de fer fran\'cais (SNCF) for track maintenance between Mommenheim and Sarreguemines (Saargem\"und) which led to the initial idea and Dagstuhl Conference 23351 during which it was worked out.

\end{document}